\documentclass[12pt,reqno]{amsart}
\usepackage{pgfplots}
\usepackage{tikz}
\pgfplotsset{compat=newest}
\pgfplotsset{ytick style={draw=none}}
\pgfplotsset{xtick style={draw=none}}
\usepackage{amsmath,amsfonts,amssymb,amsthm}
\usepackage{graphicx}
\graphicspath{ }
\usepackage{mathrsfs}
\usepackage{epstopdf}
\usepackage{csquotes}
\usepackage{wrapfig}
\usepackage{accents}
\usepackage{caption}
\usepackage{subcaption}
\usepackage{calligra}
\usepackage[colorlinks]{hyperref}
\usepackage{kantlipsum}
\usepackage{cleveref}
\hypersetup{colorlinks, 
	breaklinks,
	linkcolor=red,
	citecolor=red,
	linktocpage=true}

\numberwithin{figure}{section}

\theoremstyle{plain}
\newtheorem{thm}{Theorem}[section]
\newtheorem{lem}[thm]{Lemma}

\newtheorem{cor}{Corollary}[thm]
\theoremstyle{definition}

\theoremstyle{remark}

\usepackage{mathtools}
\makeatletter
\@namedef{subjclassname@2020}{%
	\textup{2020} Mathematics Subject Classification}
\makeatother
\title[NON-CLASSICAL GENERATING SETS IN FS GROUPS]{NON-CLASSICAL GENERATING SETS IN FUCHSIAN SCHOTTKY GROUPS}
\author[A. A. Shaikh]{Absos Ali Shaikh$^{*1}$}
\author[U. Roy]{Uddhab Roy$^{2}$}

\address{$^{1,2}$Department of Mathematics\\ The University of Burdwan\\ Burdwan--713104\\ West Bengal\\ India.}

\email{$^1$aask2003@yahoo.co.in, aashaikh@math.buruniv.ac.in}
\email{$^2$uddhabroy2018@gmail.com}

\begin{document}

\noindent\footnotetext{$^*$ Corresponding author.\\
	$\mathbf{2020}$  \textit{Mathematics Subject Classification}. Primary 20H10; Secondary 30F35.}

\maketitle

\begin{abstract}
The goal of this article is to initiate the study of estimates of the non-classical Schottky structure in the discrete subgroups of the projective special linear group over the real numbers degree $2$. In fact, in this paper, we have investigated the non-classical generating sets in the Fuchsian Schottky groups on the hyperbolic plane with boundary. A Schottky group is usually considered non-classical if the curves used in the Schottky construction are Jordan curves (except the Euclidean circles). More precisely, in this manuscript, we have provided a structure of the rank $2$ Fuchsian Schottky groups with non-classical generating sets by utilizing two suitable hyperbolic M\"obius transformations on the upper-half plane model. In particular, we have derived two non-trivial examples of Fuchsian Schottky groups with non-classical generating sets in the upper-half plane with the circle at infinity as the boundary. 
\end{abstract}

\section{\textbf{INTRODUCTION}}
In 1974, Marden (\cite{Marden}, \cite{Marden1974}) introduced the concept of non-classical Schottky groups with a non-constructive proof. In particular, Marden \cite{Marden} derived that the intersection of the closure of the classical Schottky space with the Schottky space is not the entire Schottky space, and hence there exist Schottky groups that are not classical. It is evident that a Schottky group is still classical if it is non-classical on a particular generating set. Although, it is also a very challenging task to construct a Schottky group with a non-classical generating set. In 1975, Zarrow \cite{Zarrow} claimed that he had found an example of a rank $2$ non-classical Schottky group, but later it was proved to be classical by Sato \cite{Sato}. More precisely, the Schottky group constructed by Zarrow \cite{Zarrow} was, in fact, a classical Schottky group, but on a different set of generators, and the demonstration of this was the main result in Sato's paper \cite{Sato}. In 1991, Yamamoto \cite{Yamamoto} first presented an example of the rank $2$ non-classical Schottky group in the Riemann sphere. Then, in 2009 Williams \cite{Williams} also delivered an example of the rank 2 non-classical Schottky group by applying the procedure suggested by Yamamoto in \cite{Yamamoto}. So, in the literature, the works related to the non-classical structure that have been studied in Schottky groups are all in Kleinian flavor. Indeed, in the Kleinian group theory, the non-classical structure of Schottky groups is known in the literature. However, in the Fuchsian group theory, the non-classical Schottky structure (more precisely, an explicit framework of a non-classical generating set) is still somewhat of an open problem. In $1998$, Button \cite{Button} proved that all Fuchsian Schottky groups are classical Schottky groups, but not necessarily on the same set of generators. Furthermore, Button in \cite{Button} furnished a brief example (with figure) of a Fuchsian group which is Schottky on a particular generating set, but non-classical on those generators. On the other hand, Marden (\cite{Marden}, \cite{Marden1974}) also established that all Fuchsian Schottky groups are classical. Therefore, in this direction, it is quite a difficult and fascinating problem to examine the following question.

\textbf{Question:} Does there exist any generating set for Fuchsian Schottky groups that is non-classical?

In this article, we have provided an affirmative answer to this above question by establishing Theorem $1$. 

Interestingly, it is known in the literature (more precisely, from Yamamoto's work \cite{Yamamoto}) that, the defining curves for a rank $2$ non-classical Kleinian Schottky group are two circles and two rectangles lying in the extended complex plane. Now, in the Fuchsian group, in this paper, we introduce the non-classical Schottky structure with defining curves as four hollow half-moons, i.e., four semi-circles equipped with the diameters lying on the circle at infinity (indicated by the red color arcs in Figure $2$ in Section $2$). In particular, in the Fuchsian group theory, these four hollow half-moons are the Jordan curves that represent the non-classical Schottky structure in the upper-half plane. In essence, Zarrow \cite{Zarrow} first provided the non-classical generating sets in the Kleinian Schottky group. In this article, we are going to initiate the non-classical generating sets in the Fuchsian Schottky group by establishing the following theorem. 

\begin{proof}[\textbf{Theorem 1.}] \label{t 1}
	The group $\Gamma^{N}_{S_{\kappa}}$ generated by the M\"obius transformations ${h^*}_{(3,2)_{\kappa}}$ and ${h^{**}}_{(4,1)_{\kappa}}$ is a non-classical Fuchsian Schottky group on the hyperbolic plane with the circle at infinity as the boundary for $\kappa < 10^{-11}$.
\end{proof}

In section $2$, the M\"obius transformations ${h^*}_{(3,2)_{\kappa}}$ and ${h^{**}}_{(4,1)_{\kappa}}$ are described in detail.\\ As a consequence of Theorem \ref{t 1}, in the following, we deduce Corollary \ref{c 1} and \ref{c 2} that provide the two non-trivial examples of such groups in the upper-half plane with the circle at infinity as the boundary.

\begin{cor}\label{c 1}
	The group $\Gamma^{N}_{S_{\kappa_1}}$ generated by the hyperbolic elements ${h^1}_{(3,2)_{\kappa_1}}$ and ${h^2}_{(4,1)_{\kappa_1}}$ is a non-classical Fuchsian Schottky group in the upper-half plane with the extended real numbers as the boundary for ${\kappa_1} < 4 \times 10^{-12}$.
\end{cor}

\begin{cor} \label{c 2}
	The group $\Gamma^{N}_{S_{\kappa_2}}$ generated by the hyperbolic elements ${h^3}_{(3,2)_{\kappa_2}}$ and ${h^4}_{(4,1)_{\kappa_2}}$ is a non-classical Fuchsian Schottky group in the upper-half plane model with $\mathbb{R} \cup \{ \infty \}$ as the boundary for ${\kappa_2} < 9 \times 10^{-12}$.
\end{cor}

In Section $4$, the elements ${h^1}_{(3,2)_{\kappa_1}}$, ${h^2}_{(4,1)_{\kappa_1}}$, ${h^3}_{(3,2)_{\kappa_2}}$, and ${h^4}_{(4,1)_{\kappa_2}}$ are clarified in detail.\\

   \textbf{The paper is arranged as follows:} In Section $2$, we have organized the set up for non-classical generating sets in Fuchsian Schottky groups in the hyperbolic plane. In Section $3$, we have exhibited the non-classical Schottky structure in the Fuchsian Schottky group of rank $2$ (denoted by $\Gamma^{N}_{S_{\kappa}}$) on the upper-half plane through five lemmas. Lemma \ref{l 2.1} deals with the existence of a fundamental domain for the rank $2$ classical Fuchsian Schottky group which is surrounded by four semi-circles centered on the extended real numbers. In Lemma \ref{l 2.2}, under certain assumptions, we have computed the lengths of components of the domain of discontinuity for the group $\Gamma^{N}_{S_{\kappa}}$ that intersects the real axis. In Lemma \ref{l 2.3}, we have again estimated the lengths of components of the domain of discontinuity for the group $\Gamma^{N}_{S_{\kappa}}$ which intersects the upper vertical axis. Lemma \ref{l 2.4} is related to the regions in Lemma \ref{l 2.2} and Lemma \ref{l 2.3} with the distance between the semi-circles in the set ${SC}^*_\kappa$, where ${SC}^*_\kappa = \{SC^1_{\kappa}, SC^2_{\kappa}, SC^3_{\kappa}, ..., SC^{\textbf{I}^+}_{\kappa}\}$, ($\textbf{I}^+$ denotes the positive natural number) is the complete list of the images of the semi-circles $SC_{1, \kappa}, SC_{2, \kappa}, SC_{3, \kappa}$, and $SC_{4, \kappa}$ applying by the group  $\Gamma^{N}_{S_{\kappa}}$. In Lemma \ref{l 2.5}, we have displayed that the distance between the consecutive semi-circles in the set ${SC}^*_{\kappa}$ is less than $\frac{1}{5}$. Here, we have measured the distance along the real axis or the upper imaginary axis. After proving all these lemmas, we have established the main theorem of this article in Section $4$ (see, Theorem \ref{t 1}). Section $5$ deals with the two non-trivial examples of Fuchsian Schottky groups with non-classical generating sets in the upper-half plane model (see, Corollary \ref{c 1} and \ref{c 2}).
   
   \section{\textbf{TECHNICAL CORE FOR ORGANIZING NON-CLASSICAL STRUCTURE IN FUCHSIAN SCHOTTKY GROUPS}}
   
 In \cite{ShaikhRoy21}, we have provided the structure of arbitrary finite rank classical Fuchsian Schottky groups in the hyperbolic plane with $\mathbb{R}$ $\cup$ $\{\infty\}$ as the boundary on the harmony of the real Schottky groups with two subsequent additional conditions (see, Figure $1$ in the following):
 
 $(a)$ The semi-circles are situated by the reflection of the upper imaginary axis.
 
 $(b)$ The positions of the semi-circles on the circle at infinity are non-tangential.
 \\Recall that, in this manuscript, we aim to initiate the non-classical structure (more precisely, non-classical generating sets) for the Fuchsian Schottky group of rank $2$ in the upper-half plane model.  
   
\begin{center}
	\hspace*{-0.5cm}
	\includegraphics[width=16cm, height=6cm]{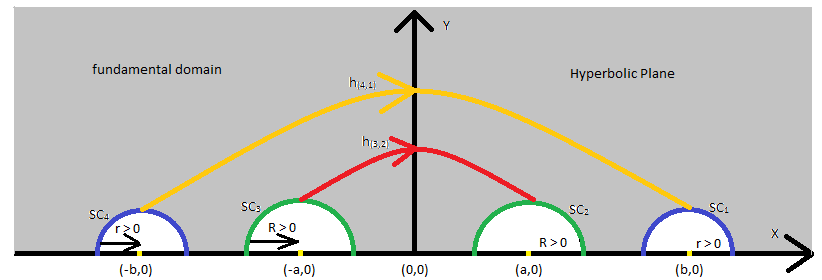} 
	$Figure: 1$
\end{center}

The preceding Figure $1$ supplies the structure of the rank $2$ classical Fuchsian Schottky group in the hyperbolic plane. Here, $SC_1$, $SC_2$, $SC_3$, and $SC_4$ indicate the semi-circles (with radius $r$ for $SC_1$, and $SC_4$; $R$ for $SC_2$, and $SC_3$) situating within the upper-half plane with centers lying on the boundary of the hyperbolic plane (i.e., circle at infinity). Here, $r$ and $R$ are the arbitrary positive real numbers. Apart from that, $h_{(4,1)}$, and $h_{(3,2)}$ represent the hyperbolic M\"obius transformations that pair the semi-circles with each other.

      It is well known that a Fuchsian group is a discrete subgroup of $PSL(2,\mathbb{R})$ (see, \cite{Katok}). On the other hand, a Schottky group is a special type of Kleinian group whereas the Kleinian groups are the discrete subgroups of $PSL(2,\mathbb{C})$ (see, \cite{Maskit1988} and \cite{Chuckrow} for details). In this paper, our goal is to set up the structure of the Fuchsian Schottky groups with non-classical generating sets by using the above two conditions $(a)$ and $(b)$ in the upper-half plane model with the circle at infinity as the boundary (see, Figure $2$ in the following). In particular, in this article, after looking at the strategy proposed by Yamamoto \cite{Yamamoto}, we explicitly prepare the framework of the rank $2$ Fuchsian Schottky groups with non-classical generating sets by imposing some new ideas in the technique operated by Yamamoto in the manuscript \cite{Yamamoto}. Firstly, for the rank $2$ Schottky group, Yamamoto \cite{Yamamoto} pointed out a suitable positive real number (viz., $\sqrt{2}$) in his used hyperbolic M\"obius transformations (one of the generators out of two) so that the group became non-classical in the extended complex plane. But in this investigation, for rank $2$ Fuchsian Schottky groups we have observed that there doesn't exist an analogous positive real number such that the groups become non-classical in the upper-half plane. Although the presence of negative real numbers is identified. For the lacuna of the existence of such a positive real number, in this study, we have taken the `$distance$' (in the Euclidean sense) between the origin and the particular number (to specify the position of that number in the hyperbolic plane, i.e., the modulus values of the numbers, without loss of generality of the construction) in the used M\"obius transformations (count as the rank of that group) to become these types of Schottky groups non-classical. Secondly, Yamamoto's construction \cite{Yamamoto} was fabricated by two circles and two rectangles in the Riemann sphere where circles were occupied by the reflection of the imaginary axis and rectangles were built up with some kind of dilation (with rotation). On the other hand, here we have approached the construction of non-classical Fuchsian Schottky groups on the upper-half plane with four semi-circles centered on the circle at infinity, and all the four semi-circles are situated by the reflections of the upper imaginary axis. Thirdly, in Yamamoto's paper \cite{Yamamoto}, the centers of the circles are lying on the real axis whereas the midpoints of the rectangles are the same and it is the origin of the axis. However, the centers of all the semi-circles in our construction belong to $\mathbb{R} - [-1, +1]$. Further, the defining curves for Yamamoto's non-classical Schottky group form the Jordan curves, whereas, in non-classical Fuchsian Schottky construction the semi-circles together with the diameters lying on the real axis create the Jordan loops look like the hollow half-moons (see, the red color curves in Figure $2$). In fact, in this manuscript, we have utilized these four hollow half-moons as the defining curves to organize the non-classical Schottky structure for Fuchsian Schottky groups in the hyperbolic plane. Fourthly, Yamamoto \cite{Yamamoto} has applied the sufficiently small positive number in one generator out of two, but the nature of the construction of Fuchsian Schottky groups demands to use of the sufficiently small positive number in both the generators to become the group non-classical which we have provided in this work. These are the four new notions that we have imposed in this paper. In this way, we have developed the literature by explicitly setting up the hollow half-moons as Jordan curves in the Schottky structure for Fuchsian Schottky groups and produced two non-trivial examples of such groups with non-classical generating sets in the upper-half plane with boundary.

Let us consider two M\"obius transformations ${h^*}_{(3,2)_{\kappa}}$ and ${h^{**}}_{(4,1)_{\kappa}}$ defined by
\begin{eqnarray}
{h^*}_{(3,2)_{\kappa}} \text{ $(\kappa \rightarrow 0^+)$ } : z & \longrightarrow & \frac{ \lambda (1- \kappa)^{-1} z + (1-\kappa)\{ \lambda ^2 (1-\kappa)^{-2} -1\}}{(1-\kappa)^{-1} z +  \lambda  (1-\kappa)^{-1}} \\ \text{  and  }
{h^{**}}_{(4,1)_{\kappa}} \text{ $(\kappa \rightarrow 0^+)$ } : z & \longrightarrow & \frac{( \lambda + 2) (1-\kappa)^{-1} z + (1-\kappa)\{( \lambda + 2)^2 (1-\kappa)^{-2} -1\}}{(1-\kappa)^{-1} z + ( \lambda + 2) (1-\kappa)^{-1}}
\end{eqnarray}
where, $z=(x+iy)\in \mathbb{H}^2$ and $ \lambda = \mod({-2})$ or $\mod(-\frac{5}{3})$, $\mod(c)$ denotes the modulus values of the number `$c$'. $\mathbb{H}^2$ represents the hyperbolic plane. 

The subsequent Figure $2$ indicates the formation of the rank $2$ non-classical Fuchsian Schottky groups in the upper-half plane model.

\begin{center}
	\hspace*{-0.1cm}
	\includegraphics[width=16.5cm, height=6.5cm]{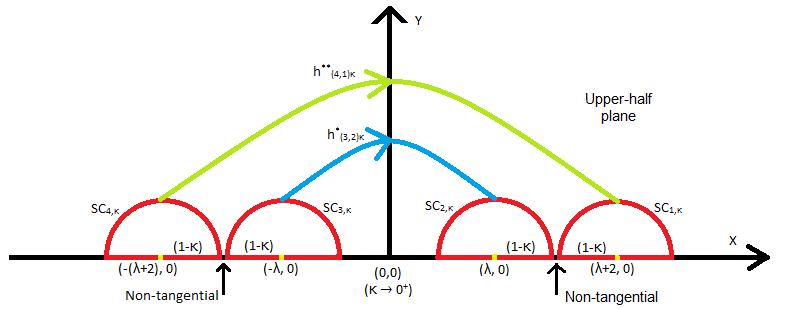} 
	$Figure: 2$
\end{center}

In the following, we set four semi-circles in the hyperbolic plane with centers lying on the boundary $\mathbb{R} \cup \{ \infty \}$.
\begin{eqnarray}
SC_{1, \kappa} &=& \{\lvert z-(\lambda + 2)\rvert = 1-\kappa\}, \notag \\
SC_{2,\kappa} &=& \{\lvert z - \lambda \rvert = 1- \kappa\}, \notag \\
SC_{3,\kappa} &=& \{\lvert z +  \lambda \rvert = 1- \kappa\}, \notag \\ \text{  and  }
SC_{4,\kappa} &=& \{\lvert z + (\lambda +2)\rvert = 1-\kappa\}. \notag 
\end{eqnarray}
Also \begin{eqnarray}
 h^*_{(3,2)_{\kappa}}(SC_{3,\kappa}) &=& SC_{2,\kappa}, \notag \\ \text{ and } {h^{**}}_{(4,1)_{\kappa}}(SC_{4,\kappa}) &=& SC_{1,\kappa}.   \notag
\end{eqnarray}

 In this article, we have used an innovative idea on the second M\"obius transformation, i.e., the map ${h^{**}}_{(4,1)_{\kappa}}$ in the ensuing way.
 
 We take, for $\kappa \rightarrow 0^+$,   $\Big(\frac{( \lambda  + 2)z + \{ ( \lambda  + 2)^2 - (1 - \kappa)^2\}}{z + ( \lambda  + 2)} \Big) < ( \lambda  + 2)$,  $\forall$  $z$  $\in$ $\mathbb{H}^2$ in the sense of the modulus values (i.e., the distance measurement) of the complex numbers in the upper-half plane, where $\mathbb{H}^2$ denotes the hyperbolic plane.\\
 For example, let, $z$ $=$ $x+iy$, where, $x$ $\in$ $\mathbb{R}$, $y$ $\in$ $\mathbb{R}^+$.\\ Then
  $\Big(\frac{( \lambda  + 2)z + \{ ( \lambda  + 2)^2 - (1 - \kappa)^2\}}{z + ( \lambda  + 2)} \Big)$ converts to  $\Big(\frac{( \lambda  + 2) \lvert x+ iy \rvert + \{ ( \lambda  + 2)^2 - (1 - \kappa)^2\}}{\lvert x+ iy \rvert + ( \lambda  + 2)} \Big)$ which is less than $(\lambda  + 2)$.\\

        \section{\textbf{PROOF OF THE LEMMAS}}
    
    To establish Theorem \ref{t 1}, we have taken the help of five lemmas which we derived in this section.

   At first, for the convenience of the readers, we aim to deliver the existence of a fundamental domain of the rank $2$ classical Fuchsian Schottky group in the upper-half plane model. The subsequent lemma will be effective in this direction.

   \begin{lem} \label{l 2.1}
   Let $\Gamma_{S_{2}}$ be a rank $2$ (classical) Fuchsian Schottky group generated by two M\"obius hyperbolic transformations. Suppose, $h_s$ is an element of $\Gamma_{S_{2}}$. Then there exists a fundamental domain $F_{\Gamma_{S_{2}}}$ for $\Gamma_{S_{2}}$ enclosed by four semi-circles with the center on the circle at infinity, where one of the semi-circles separates the fixed points of the transformation $h_s$.
   \end{lem}
 \begin{center}
              \hspace*{-.1cm}
              \includegraphics[width=17cm, height=5.9cm]{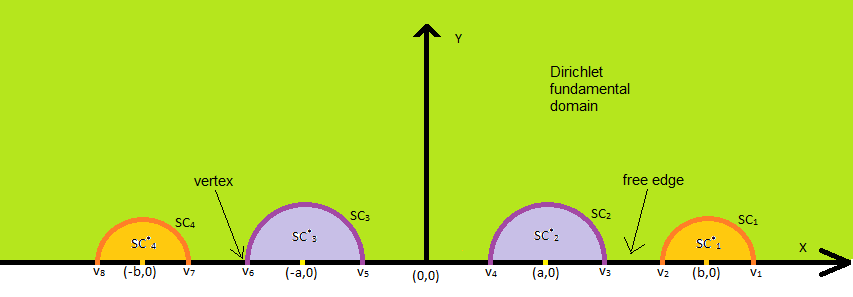} 
              $Figure: 3$
              \end{center}

   \begin{proof}
   Suppose $F_{\Gamma_{S_{2}}}$ (see, the light green color region in Figure $3$) is the fundamental domain  of $\Gamma_{S_{2}}$ encircled by four semi-circles $SC_{1}, SC_{2}, SC_{3}$, and $SC_{4}$, ($SC_{4}$ is paired with $SC_{1}$, whereas $SC_{3}$ twins with $SC_{2}$) containing four free edges. Note that, all the vertices (viz., $v_1$, $v_2$, $v_3$, $v_4$, $v_5$, $v_6$, $v_7$, and $v_8$) of the polygon that represents the Dirichlet fundamental region are all improper vertices (see, \cite{Beardon} for the definition of improper vertex). Observe that, here, all vertices are infinite vertices of the overhead polygon. Let, $\{\infty \} \in \Gamma_{S_{2}}$, as an ordinary point of the group. Obviously, $F_{\Gamma_{S_{2}}}$ produces an infinite area unbounded region in the upper-half plane. In \cite{ShaikhRoy21}, we have seen that the group $\Gamma_{S_{2}}$ is a discrete subgroup of $PSL(2,\mathbb{R})$ acting in the hyperbolic plane $\mathbb{H}^2$ $=\{(x,y): x \in \mathbb{R},$ $y>0\}$. Let, ${SC}^*_1$, ${SC}^*_2$,  ${SC}^*_3$, and ${SC}^*_4$ be the $4$ half disks (all are semi-circular in shape) in $\mathbb{H}^2 \cup \{\infty \}$ centered on the extended real numbers with boundary $SC_{1}, SC_{2}$, $SC_{3}$, and $SC_{4}$ respectively perpendicular to $\mathbb{R} \cup \{\infty\}$. Then the unbounded region ${F}^*_{\Gamma_{S_{2}}}$ of $\mathbb{H}^2$ is enclosed by four semi-circular disks ${SC}^*_1$, ${SC}^*_2$,  ${SC}^*_3$, and ${SC}^*_4$ is a fundamental domain for $\Gamma_{S_{2}}$ in $\mathbb{H}^2 \cup \{\infty \}$. Since ${h}_s$ is hyperbolic, we can consider a point `${\eta}$' on the axis of the element ${h}_s$. So, there exists an element ${g}_s \in {\Gamma}_{{S_2}}$ such that ${g}_s ({\eta}) \in {F}^*_{\Gamma_{S_{2}}}$. Note that, in the hyperbolic plane the region ${F}^*_{\Gamma_{S_{2}}}$ is geodesically convex. Hence, one of the semi-circular half disks ${SC}^*_1$, ${SC}^*_2$,  ${SC}^*_3$, and ${SC}^*_4$, say ${SC}^*_1$, separates the fixed points of $g_s h_s (g_s)^{-1}$. Therefore, the boundary of the half disk ${SC}^*_1$, i.e., the semi-circular curve ${SC}_1$ will also separate these fixed points. So the required fundamental domain of $\Gamma_{S_{2}}$ is the zone surrounded by $(g_s)^{-1}(SC_{1})$, $(g_s)^{-1}(SC_{2})$, $(g_s)^{-1} (SC_{3})$, and $(g_s)^{-1}(SC_{4})$ which is basically the infinite area region $F_{{\Gamma}_{S_2}}$ in the upper-half plane model. This completes the proof.
   \end{proof} 
   To prove the remaining four lemmas, our objective is now to set up the background notations which will be operated throughout the rest of this section. Our goal is to prove the required Theorem \ref{t 1} by the method of contradiction. Suppose that for $\kappa < 10^{-11}$, the Fuchsian Schottky group $\Gamma^{N}_{S_{\kappa}}$ is classical. Since we have already shown the existence of a fundamental domain for the rank $2$ classical Fuchsian Schottky group in the above lemma, let us assume that $F_{\Gamma^{N}_{S_{\kappa}}}$ be that fundamental domain for the group $\Gamma^{N}_{S_{\kappa}}$ sketched by four semi-circles $SC_{1,\kappa}$, $SC_{2,\kappa}$, $SC_{3,\kappa}$, and $SC_{4,\kappa}$, one of which, say $SC_{4,\kappa}$, separates the fixed points $\tau$ and $\{-\tau\}$ $\textbf{(}$where, $ \tau = \sqrt{(\lambda +1)(\lambda +3)}\textbf{)}$ of the transformation ${h^{**}}_{(4,1)_{\kappa}}$, and so does another boundary component of $F_{\Gamma^{N}_{S_{\kappa}}}$ other than the semi-circle $SC_{4,\kappa}$. Consider the set ${SC}^*_\kappa = \{SC^1_{\kappa}, SC^2_{\kappa}, SC^3_{\kappa}, ..., SC^{\textbf{I}^+}_{\kappa}\}$ which is the complete list of the images of the semi-circles $SC_{1,\kappa}$, $SC_{2,\kappa}$, $SC_{3,\kappa}$, and $SC_{4,\kappa}$ applying by the group $\Gamma^{N}_{S_{\kappa}}$ that are nested and intersect the real interval $(-\tau,$  $\lambda + 3)$ once, satisfying one of the following cases: \\ $(i)$ Each semi-circle $SC^j_{\kappa}$ segregates the fixed points $ \tau$ and $\{- \tau\}$ of the transformation ${h^{**}}_{(4,1)_{\kappa}}$, \\ $(ii)$ Each semi-circle $SC^{j+1}_{\kappa}$ keeps apart the semi-circle $SC^j_{\kappa}$ from the fixed point $\{-\tau \}$, and \\ $(iii)$ The semi-circles $SC^1_{\kappa},$ $SC^j_{\kappa}$, $j=2,3, ..., (\textbf{I}^+ -1)$ and $SC^{\textbf{I}^+}_{\kappa}$ meet the intervals $\textbf{[}\tau,$  $(\lambda + 3) \textbf{)},$ $\textbf{(}(\lambda + 1),$  $\tau \textbf{)}$, and $\textbf{(}-\tau,$  $(\lambda + 1) \textbf{]}$ respectively, where $\textbf{I}^+$ denotes the positive natural number.
   
    Observe that, if two semi-circles in the set ${SC}^*_\kappa$ (say, ${SC}^{m_1}$ and ${SC}^{m_2}$, where $m_1,m_2 \in \textbf{I}^+$) are not separated by other semi-circles in the set ${SC}^*_\kappa$, these two semi-circles occupy the boundary of $F_{\Gamma^{N}_{S_{\kappa}}}$. Now, the semi-circles $SC^j_{\kappa}$ and $SC^{j+1}_{\kappa}$ together with any two semi-circles in the family of curves $\{q(SC_{i,{\kappa}}): q\in \Gamma^{N}_{S_{\kappa}}, i=1,2,3,4\}$ surround a fundamental domain for the group $\Gamma^{N}_{S_{\kappa}}$. Suppose that $Y_{jk{\kappa}} = \{SC^j_{\kappa} \cap i^{k-1} \mathbb{R}^+\}$, where $k=1,2,3,$ and $\mathbb{R}^+ = \{x \in \mathbb{R}: x>0\}$. In the hyperbolic plane, we set the largest distance between two consecutive semi-circles in the set ${SC}^*_\kappa$ as $Z_{j{\kappa}}=\max_k \{\lvert Y_{jk{\kappa}} - Y_{(j+1)k{\kappa}} \rvert\}$. Assume that $\Omega(\Gamma^{N}_{S_{\kappa}})$ is the domain of discontinuity of the group $\Gamma^{N}_{S_{\kappa}}$.
   
  After organizing the background notations in the preceding, we aim now to establish Lemma \ref{l 2.3}. But before that it is essential to prove the ensuing lemma. In particular, we have deduced the following Lemma \ref{l 2.2} under certain supposition. The succeeding lemma is related to the lengths of components of the domain of discontinuity for the group $\Gamma^{N}_{S_{\kappa}}$ that intersect the boundary of the upper-half plane.
   
   \begin{lem} \label{l 2.2}
      The length of each component of $\{\Omega(\Gamma^{N}_{S_{\kappa}}) \cap \mathbb{R}\}$ which meets the region made up of the union of intervals given as $[-(\lambda + 3), -(\lambda + 1)] \cup [(\lambda + 1), (\lambda + 3)] $ is less than the subsequent quantity \begin{eqnarray}
       \Big\{\frac{2}{(4 \lambda^2 + 9\lambda + 3)} \times \sqrt{2.01} \times \sqrt{\kappa} \times \sqrt{6 \lambda^4 + 16 \lambda^3 + 3 \lambda^2 - 16 \lambda -8}\Big\}, \notag
      \end{eqnarray} where $\lambda = \lvert x \rvert > 1$ and $\lvert x \rvert$ denotes the modulus value of the number $x$ .
      \end{lem}  
      
      \begin{proof}
      First, we consider two functions defined by
      $${h^*}_{(3,2)_{\kappa}} (z) = \frac{ x (1-\kappa)^{-1} z + (1-\kappa)\{ x ^2 (1-\kappa)^{-2} -1\}}{(1-\kappa)^{-1} z +  x  (1-\kappa)^{-1}}$$
         and
         $${h^{**}}_{(4,1)_{\kappa}} (z) = \frac{( x + 2) (1-\kappa)^{-1} z + (1-\kappa)\{( x + 2)^2 (1-\kappa)^{-2} -1\}}{(1-\kappa)^{-1} z + ( x + 2) (1-\kappa)^{-1}}$$
         where $x$ is a positive real number greater than $1$.
         
         Let, $\{{h^*}_{(3,2)_{\kappa}} {h^{**}}_{(4,1)_{\kappa}} {h^*}_{(3,2)_{\kappa}} ({h^{**}}_{(4,1)_{\kappa}})^{-1}\}(z)$ $=$ $D(z)$ (say).\\ This gives the fixed points as ${z_i}$, for $i=1,2,$ where
        $${z_i} = \frac{1}{(4x^2 + 9x + 4 - 2x \kappa +x \kappa^2 )} \times [2 \{ 2x^3 + 6x^2 + 5x + 1 + 2(x + 1)\kappa - (x + 1)\kappa^2\} \pm  \textbf{\{} (9x^4 + 48 x^3 + 91 x^2 $$ 
           $$ + 72x + 20) + 2 \kappa (6 x^4 + 16 x^3 + 3 x^2 - 16x -8) + \kappa^2(-2 x^4 - 16 x^3 + x^2 + 48x + 24)$$ 
           $$+4 \kappa^3 (-x^4 + x^2 - 8x -4) +  \kappa^4 (x^4 - 11 x^2 + 8x + 4) + 6 \kappa^5 (x^2) - 2 \kappa^6 (x^2) \textbf{\}}^{\frac{1}{2}}] $$ 
         Without loss of generality, we assume that, $9x^4 + 48 x^3 + 91 x^2 + 72x + 20 = 0$. \\ This equation yields the subsequent roots. 
         
          $x = -2, -\frac{5}{3}, -1,$ and $-\frac{2}{3}$. \\ Now, let $ \lambda = \lvert x  \rvert,$ where $\lvert x  \rvert$ denotes the modulus value of the numbers  $-2$ and  $-\frac{5}{3}$, i.e., $\lambda =2$ and $1.6666666667$ (approx.) $(>1)$.
        
        We have supposed $x$ to be a positive real number greater than $1$ but the equation $9x^4 + 48 x^3 + 91 x^2 + 72x + 20 = 0$ could not give such a real number as a root. On the other hand, if we want to use the technique from Yamamoto's paper \cite{Yamamoto}, we have to take $9x^4 + 48 x^3 + 91 x^2 + 72x + 20$ equal to zero. Instead, we could not find our required sufficiently small positive delta. Consequently, for simplicity, we have taken the modulus values $(>1)$ of the roots of that equation without loss of generality. 
         
        Let, ${\psi_1}$ be the component of $\{\Omega(\Gamma^{N}_{S_{\kappa}}) \cap \mathbb{R}\}$ bounded by the fixed points of the transformation  $\{{h^*}_{(3,2)_{\kappa}} {h^{**}}_{(4,1)_{\kappa}} {h^*}_{(3,2)_{\kappa}} ({h^{**}}_{(4,1)_{\kappa}})^{-1}\}$. All components of $\{ \Omega(\Gamma_{S_{\kappa}}^{N}) \cap \mathbb{R}\}$ are equivalent to one another for the group $\Gamma^{N}_{S_{\kappa}}$. We set ${\psi}_2 = ({h^*}_{(3,2)_{\kappa}})^{-1}({\psi}_1)$, ${\psi}_3 = ({h^{**}}_{(4,1)_{\kappa}})^{-1}({\psi}_2)$, and ${\psi}_4 = ({h^{**}}_{(4,1)_{\kappa}})^{-1}({\psi}_1)$. Let, ${\Psi}$ be a component of $\{\Omega(\Gamma_{S_{\kappa}}^{N}) \cap \mathbb{R}\}$ included in one of the semi-circles $SC_{1,\kappa}$ and $SC_{3,\kappa}$ or $SC_{2, \kappa}$ and $SC_{4,{\kappa}}$. 
        
        So, ${\Psi}$ can be written as in the subsequent way. 
        \begin{eqnarray}
       \beta_{2t}({\Psi}) = \{({h^*}_{(3,2)_{\kappa}})^{s_{2t}}({h^{**}}_{(4,1)_{\kappa}})^{s_{2t-1}}.({h^*}_{(3,2)_{\kappa}})^{s_{2t-1}}({h^{**}}_{(4,1)_{\kappa}})^{s_{2t-2}} ... ({h^*}_{(3,2)_{\kappa}})^{s_3} ({h^{**}}_{(4,1)_{\kappa}})^{s_2}.({h^*}_{(3,2)_{\kappa}})^{s_2} \notag\\
                ({h^{**}}_{(4,1)_{\kappa}})^{s_1}\}({\Psi}) \notag,
        \end{eqnarray}   
         where, ${\Psi}$ is either ${\psi_1}, {\psi_2}, {\psi_3}$, or ${\psi_4}$ and $s_{2t} s_{2t-1} ... s_2 \ne 0$. \\ We denote $l({\Psi})$ to be the length of ${\Psi}$. 
        \\Then we obtain \begin{eqnarray} \text{ $l({\Psi}$}) &<& \Big\{ \frac{1}{(4x^2 + 9x + 4 - 2x \kappa +x \kappa^2 )} \times \sqrt{2.01} \times \sqrt{\kappa} \times \sqrt{6 \lambda^4 + 16 \lambda^3 + 3 \lambda^2 - 16 \lambda -8}\Big\}\notag \\ 
        &<& \Big\{\frac{2}{(4 \lambda^2 + 9\lambda + 3)} \times \sqrt{2.01} \times \sqrt{\kappa} \times \sqrt{6 \lambda^4 + 16 \lambda^3 + 3 \lambda^2 - 16 \lambda -8}\Big\}\notag.
        \end{eqnarray} Now, we aim to show that the length of any image of ${\Psi}$ under $\beta_{2t}$ is less than the length of ${\Psi}$, which we have established by proving that $\lvert \beta'_{2t} (x) \rvert < 1,$ $\forall$ $x$ $\in$ ${\Psi}$. In fact, we have succeeded by utilizing the method of induction concerning $t$.
        
        First of all, for $t=1$, we get $\beta_2 = \{({h^*}_{(3,2)_{\kappa}})^{s_2} ({h^{**}}_{(4,1)_{\kappa}})^{s_1}\}$.\\  In the following, we have investigated three cases for the value of $s_1$. \\ (i) Case I, when $s_1 = 0,$ (ii) Case II, for $s_1 <0,$ and (iii) Case III, when $s_1 >0$. 
        
         Clearly, \begin{eqnarray}
         \lvert ({h^*}_{(3,2)_{\kappa}})' (x) \rvert &=& \Big \lvert \frac{(1-\kappa)^2}{(x+ \lambda)^2} \Big \rvert \notag, \\  
         \text{  and  }
        \lvert ({h^{**}}_{(4,1)_{\kappa}})' (x)\rvert &=& \Big \lvert \frac{(1-\kappa)^2}{\{x+(\lambda+2)\}^2} \Big \rvert. \notag
         \end{eqnarray} Inside the semi-circle $SC_{3, \kappa}$, we get  $\lvert ({h^*}_{(3,2)_{\kappa}})' (x) \rvert > 1$ since $\lvert x + \lambda \rvert < 1 - \kappa$ and outside the semi-circle $SC_{3, \kappa}$, we obtain  $\lvert ({h^*}_{(3,2)_{\kappa}})' (x) \rvert < 1$ since $\lvert x + \lambda \rvert > 1 - \kappa$. However, inside the semi-circle $SC_{4,\kappa}$, we have $\lvert ({h^{**}}_{(4,1)_{\kappa}})' (x)\rvert > 1$ since $\lvert x + (\lambda+2) \rvert < 1 - \kappa$ and outside the semi-circle $SC_{4,\kappa}$, we get $\lvert ({h^{**}}_{(4,1)_{\kappa}})' (x)\rvert < 1$ since $\lvert x + (\lambda+2) \rvert > 1 - \kappa$. \\ $(i)$ Case I, for $s_1 =0:$ In the first case, our goal is to prove that $\lvert \{({h^*}_{(3,2)_{\kappa}})^{s_2}\}' (x)  \rvert < 1$ $\forall$ $x$ $\in$ ${\Psi}$ by using the method of induction on $s_2$. Now if we observe the case $s_2 > 0$, for $s_2 = 1$ the induction becomes trivial. For $s_2 = 2$, we obtain the following. \begin{eqnarray}
           \lvert \{({h^*}_{(3,2)_{\kappa}})^2\}'(x)\rvert &=& \frac{d}{dx} \Big[ \Big(\frac{ \lambda x + \{ \lambda ^2 - (1-\kappa)^{2}\}}{x +  \lambda}\Big)^2 \Big] \notag \\
           &=& \frac{d}{dx} \Big[\Big(\lambda - \frac{(1-\kappa)^2}{x+\lambda}\Big)^2 \Big], \text{ since }  x \ne - \lambda \notag \\
            &=& 2 \Big(\lambda - \frac{(1-\kappa)^2}{x+\lambda}\Big)\Big(\frac{1-\kappa}{x+\lambda} \Big)^2 \notag \\
           &<& 2 \Big(\lambda - \frac{(1-\kappa)^2}{x+\lambda}\Big),\notag
           \end{eqnarray} since outside the semi-circle $SC_{3, \kappa}$, we have $ (1 - \kappa) < ( x + \lambda )$. 
           \\ Now, if we take \begin{eqnarray}
           x &<& \frac{1}{4}\{(-2\lambda + 1) - \sqrt{(2\lambda + 1)^2 + 16 (1-\kappa)^2}\}, \notag \\
           \text{ and } x &>& \frac{1}{4}\{(-2\lambda + 1) + \sqrt{(2\lambda + 1)^2 + 16 (1-\kappa)^2}\}, \notag
           \end{eqnarray} we get \begin{eqnarray}
           \lvert \{({h^*}_{(3,2)_{\kappa}})^2\}'(x)\rvert < 1.\notag
           \end{eqnarray} Note that, if we consider outside the restriction of $x$, i.e., \begin{eqnarray}
           \frac{1}{4}\{(-2\lambda + 1) - \sqrt{(2\lambda + 1)^2 + 16 (1-\kappa)^2}\} < x < \frac{1}{4}\{(-2\lambda + 1) + \sqrt{(2\lambda + 1)^2 + 16 (1-\kappa)^2}\}, \notag
           \end{eqnarray} the position of $\psi_1$ will be on the right of this domain whereas $\psi_2$ situated to the left of that region, and $\psi_3$ is located to the right of this zone. Further, for some very small $\kappa$, $\psi_4$ is also outside that range. Hence, we obtain $\forall$ $x$ $\in$ $\Psi$, $\lvert \{({h^*}_{(3,2)_{\kappa}})^{s_2}\}' (x)  \rvert < 1$. Now, our goal is to prove that the condition $\lvert \{({h^*}_{(3,2)_{\kappa}})^{s_2}\}' (x)  \rvert < 1$ is true for $s_2 = n+1$ when it happens for $s_2 = n$. 
           
           This can be easily seen from the following argument.
           \begin{eqnarray}
           \lvert \{  ({h^*}_{(3,2)_{\kappa}})^{n+1} \}' (x) \rvert &=& \Big \lvert \frac{d}{dx} \{ ( {h^*}_{(3,2)_{\kappa}} )^{n+1}  (x) \} \Big \lvert \notag \\ &=&\Big \lvert \frac{d\{({h^*}_{(3,2)_{\kappa}})^{n+1} (x)\}}{d\{({h^*}_{(3,2)_{\kappa}})^n (x)\}} \Big \rvert \times \Big \lvert \frac{d\{({h^*}_{(3,2)_{\kappa}})^{n} (x)\}}{dx} \Big \rvert \notag \\
           &<& \Big \lvert \frac{d\{({h^*}_{(3,2)_{\kappa}})^{n+1} (x)\}}{d\{({h^*}_{(3,2)_{\kappa}})^n (x)\}} \Big \rvert, \text{ from our assumption } \Big \lvert \frac{d\{({h^*}_{(3,2)_{\kappa}})^{n} (x)\}}{dx} \Big \rvert < 1 \notag \\
           &=& \Big \lvert \frac{d\{({h^*}_{(3,2)_{\kappa}})(x)\}}{dx} \Big \rvert, \text{ replacing x by } ({h^*}_{(3,2)_{\kappa}})^{n} (x) \notag \\
           &=& \Big \lvert \frac{(1-\kappa)^2}{(x + \lambda)^2} \Big \rvert \notag \\
           &<& \frac{1}{\lvert x + \lambda \rvert^2}\notag \\ 
           &=& \frac{1}{\lvert( {h^*}_{(3,2)_{\kappa}} )^{n}  (x) + \lambda \rvert^2}\notag.   
           \end{eqnarray}
           Now, if we take the modulus values of the coordinates that are coming from $({h^*}_{(3,2)_{\kappa}} )^{n}  (x)$ when $x \in \psi_i$, for $i = 1, 2, 3,$ and $4$, which run inside the semi-circle $SC_{2, \kappa}$ and utilize that modulus values in $\frac{1}{\lvert( {h^*}_{(3,2)_{\kappa}} )^{n}  (x) + \lambda \rvert^2}$, we obtain 
           \begin{eqnarray}
           \lvert \{  ({h^*}_{(3,2)_{\kappa}})^{n+1} \}' (x) \rvert < 1 \notag.
           \end{eqnarray} 
           
            Hence, we reach the required one.

           The situation is very similar for the case when $s_2 < 0$. First, we show it for $s_2 = -2$, after that, our target will be to prove that assuming $\lvert \{({h^*}_{(3,2)_{\kappa}})^{s_2}\}' (x)  \rvert < 1$ is true for $s_2 = n$ (where, $n<0)$, then it is also true for $s_2 = n-1$. In this manner for $s_2 < 0$ one can follow the exact method as we have discussed in the previous case. Therefore, Case I is done. \\ (ii) Case II, when $s_1 < 0:$ We have established this case by using a similar type of induction method as in the preceding one, where one can write the derivative of $\beta_2$ in the ensuing way. 
           
           \begin{eqnarray}
           \lvert (\beta_2)'(x) \rvert &=& \lvert \{ ({h^*}_{(3,2)_{\kappa}})^{s_2} ({h^{**}}_{(4,1)_{\kappa}})^{s_1}\}' (x) \rvert \notag \\
           &=& \Big \lvert \frac{d}{dx} \{({h^*}_{(3,2)_{\kappa}})^{s_2} ({h^{**}}_{(4,1)_{\kappa}})^{s_1} (x) \} \Big \rvert \notag \\
           &=& \Big \lvert \frac{d\{({h^*}_{(3,2)_{\kappa}})^{s_2} ({h^{**}}_{(4,1)_{\kappa}})^{s_1} (x) \}}{d\{({h^{**}}_{(4,1)_{\kappa}})^{s_1} (x) \}} \Big \rvert \times \Big \lvert \frac{d\{({h^{**}}_{(4,1)_{\kappa}})^{s_1} (x) \}}{dx} \Big \rvert \notag \\
           &<& \frac{1}{\lambda+2} \times \Big \lvert \frac{d\{({h^*}_{(3,2)_{\kappa}})^{s_2} ({h^{**}}_{(4,1)_{\kappa}})^{s_1}(x)\}}{d\{({h^{**}}_{(4,1)_{\kappa}})^{s_1}(x)\}}  \Big \rvert \notag \\
           &<& \Big \lvert \frac{d\{({h^*}_{(3,2)_{\kappa}})^{s_2} ({h^{**}}_{(4,1)_{\kappa}})^{s_1} (x) \}}{d\{({h^{**}}_{(4,1)_{\kappa}})^{s_1} (x) \}} \Big \rvert \notag \\
           &=& \Big \lvert \frac{d\{({h^*}_{(3,2)_{\kappa}})^{s_2} (v)\}}{dv}  \Big \rvert \notag, \text{ replacing } v \text{ by } \{({h^{**}}_{(4,1)_{\kappa}})^{s_1}(x)\}\\
           &=&  \frac{\lvert 1-\kappa \rvert^2}{\lvert v+ \lambda \rvert^2}  \notag \\
           &<&  \frac{1}{\lvert v+ \lambda \rvert^2}  \notag \\
           &<& 1, \notag 
           \end{eqnarray}
            by using Case I when  $s_1 = 0$  and applying all the regions $(\text{for } s_1 < 0)$ that are outside the semi-circle $SC_{3, \kappa}$.

            Hence, the case for $s_1 <0$ follows. 
         \\ $(iii)$ Case III, for $s_1>0:$ For the first step of the induction concerning $t$, we aim to derive that for $s_1 >0$, the image of ${\Psi}$ under the transformations $\{({h^*}_{(3,2)_{\kappa}})^{s_2} ({h^{**}}_{(4,1)_{\kappa}})^{s_1} \}$ is less than ${\Psi}$. Now, for the subcases, we first examine the case for $s_2 > 0$. Observe that, if $s_1 = 1$ then from our assumption, we get that the transformation $({h^{**}}_{(4,1)_{\kappa}})$ sends ${\psi}_3$ to ${\psi}_2$ and ${\psi}_4$ to ${\psi}_1$, and also we already have $\lvert \{({h^*}_{(3,2)_{\kappa}})^{s_2}\}'(x)\rvert < 1$ $\forall$ $x$ $\in$ ${\Psi}$. Further, when $s_1 >1$, the images of ${\psi}_1$ and ${\psi}_2$ under the transformation $({h^{**}}_{(4,1)_{\kappa}})^{s_1}$ are outside the semi-circle $SC_{3, \kappa}$. Again, outside the semi-circle $SC_{3, \kappa}$, the M\"obius transformation $({h^*}_{(3,2)_{\kappa}})$ is contracting, since we get \begin{eqnarray}
        \lvert ({h^*}_{(3,2)_{\kappa}})' (x) \rvert &=& \Big \lvert \frac{(1-\kappa)^2}{(x+ \lambda)^2} \Big \rvert \notag \\ 
        &<& 1, \notag
        \end{eqnarray}  for $\lvert x + \lambda \rvert > (1 -\kappa)$, the points that are outside the semi-circle $SC_{3, \kappa}$. Therefore, when $s_2 =1$, the transformation $\{({h^*}_{(3,2)_{\kappa}})^{s_2} ({h^{**}}_{(4,1)_{\kappa}})^{s_1} \}$ provides the longest interval. So, our purpose is now to deduce that the transformation $\{({h^*}_{(3,2)_{\kappa}})^{s_2} ({h^{**}}_{(4,1)_{\kappa}})^{s_1} \}({\Psi})$ is shorter than ${\Psi}$, $\forall$ $s_1$. To establish that we aim to show that $B_1$ and $B_2$ are less than $\lvert {\psi}_{1a} - {\psi}_{1b} \rvert$, where $B_1$ and $B_2$ are described in the following.  
        \begin{eqnarray}
        B_1 &=& \Big \lvert ({h^*}_{(3,2)_{\kappa}}) \Big \{ \Big(\frac{(\lambda+2){\psi}_{1a} + A}{{\psi}_{1a} + (\lambda+2)} \Big)^{s_1} \Big\} - ({h^*}_{(3,2)_{\kappa}}) \Big \{ \Big(\frac{(\lambda+2){\psi}_{1b} + A}{{\psi}_{1b} + (\lambda+2)} \Big)^{s_1} \Big \} \Big \rvert \notag, \\ \text{  and  }
        B_2 &=& \Big \lvert ({h^*}_{(3,2)_{\kappa}}) \Big \{ \Big(\frac{(\lambda+2){\psi}_{2a} + A}{{\psi}_{2a} + (\lambda+2)} \Big)^{s_1} \Big\} - ({h^*}_{(3,2)_{\kappa}}) \Big \{ \Big(\frac{(\lambda+2){\psi}_{2b} + A}{{\psi}_{2b} + (\lambda+2)} \Big)^{s_1} \Big \} \Big \rvert \notag                
        \end{eqnarray}
          where, $A = \{(\lambda+2)^2 - (1-\kappa)^2\}$ and ${\psi}_{1a}, {\psi}_{1b}, {\psi}_{2a},$ and ${\psi}_{2b}$ are the end points of ${\psi}_1$ and ${\psi}_2$. \\ Now \begin{eqnarray}
          B_1 &=& \frac{(1-\kappa)^2 (M^{s_1} - N^{s_1})}{(M^{s_1} + \lambda) (N^{s_1} + \lambda)} \notag \\
          &<& \frac{ (M^{s_1} - N^{s_1})}{(M^{s_1} + \lambda) (N^{s_1} + \lambda)} \notag
          \end{eqnarray} 
            where, $M = \Big(\frac{(\lambda+2){\psi}_{1a} + \{(\lambda+2)^2 - (1-\kappa)^2\}}{{\psi}_{1a} + (\lambda+2)} \Big)$ and $N=\Big(\frac{(\lambda+2){\psi}_{1b} + \{(\lambda+2)^2 - (1-\kappa)^2\}}{{\psi}_{1b} + (\lambda+2)} \Big) $. \\Here, we have looked at $4$ cases as given below (for $\kappa$ tends to $0^+$): \\ $(i)$ $s_1 \rightarrow 0$, $B_1$ $\rightarrow 0$. \\ $(ii)$ $s_1 \rightarrow \infty$, $B_1$ $\rightarrow 0$. \\ $(iii)$  $s_1 \rightarrow$ absolute value less than $1$, $B_1$ $\rightarrow 0$. \\ $(iv)$ $s_1 = 1$, \begin{eqnarray}
            B_1 &<& \frac{{\psi}_{1a} -{\psi}_{1b} }{({\psi}_{1a} + (\lambda+2))({\psi}_{1b} + (\lambda+2))} \notag \\
            &<& ({\psi}_{1a} -{\psi}_{1b}). \notag
            \end{eqnarray} 
                  For $B_2$, the argument is similar if we replace ${\psi}_{1a}, {\psi}_{1b}$ with  ${\psi}_{2a}$ and  ${\psi}_{2b}$. Therefore, any image of ${\Psi}$ under $\{({h^*}_{(3,2)_{\kappa}})^{s_2} ({h^{**}}_{(4,1)_{\kappa}})^{s_1} \}$ is less than ${\Psi}$, for $s_1 > 0$. Hence, we have shown the first step of the induction concerning $t$. Now, our objective is to establish the next step of the induction. For this purpose, we are aiming to show that for any combination of $s_1$ and $s_2$, we get $\lvert (\beta_2)'(x) \rvert < 1$ for any $x \in {\Psi}$. So, for assuming 
                  \begin{eqnarray}
                  \lvert \{({h^*}_{(3,2)_{\kappa}})^{s_{2n+2}} ({h^{**}}_{(4,1)_{\kappa}})^{s_{2n+1}} ({h^*}_{(3,2)_{\kappa}})^{s_{2n}} ({h^{**}}_{(4,1)_{\kappa}})^{s_{2n-1}} ... ({h^*}_{(3,2)_{\kappa}})^{s_2} ({h^{**}}_{(4,1)_{\kappa}})^{s_1}\} \rvert < 1, \forall \text{ x } \in \Psi \notag,
                  \end{eqnarray} 
                  we have to derive that it is true for $t=(n+1)$, i.e., the image of ${\Psi}$ under the transformation $\beta_{2(n+1)}$ is not greater than the image of ${\Psi}$ under $\beta_{2n}$. We have shown this by proving the two subcases $(i)$ $s_{2n+1} < 0$ and $(ii)$ $s_{2n+1} >0$ separately. \\ $(i)$ Subcase I, for $s_{2n+1} < 0$: Recall that, the image of the length of ${\Psi}$ under $\beta_{2n}$ is less than the length of ${\Psi}$. Since $s_{2n+1} < 0$, we get that the image of the length of ${\Psi}$ under $\{({h^{**}}_{(4,1)_{\kappa}})^{s_{2n+1}} \beta_{2n}\}$ is again less than the length of ${\Psi}$. Also since the image of $\Psi$ under the transformation $\beta_{2n}$ can be inside either the semi-circle $SC_{2,\kappa}$ or $SC_{3, \kappa}$, we assert that the image under the transformation $\{({h^{**}}_{(4,1)_{\kappa}})^{s_{2n+1}} \beta_{2n}\}$ is inside the semi-circle $SC_{4, \kappa}$, i.e., outside the semi-circles $SC_{2,\kappa}$ and $SC_{3,\kappa}$. Again, since the transformation $({h^*}_{(3,2)_{\kappa}})$ is contracting outside the semi-circle $SC_{2,\kappa}$, $({h^*}_{(3,2)_{\kappa}})^{-1}$ is contracting outside of the region $SC_{3,\kappa}$. So, the image under the transformation $\{({h^*}_{(3,2)_{\kappa}})^{s_{2n+2}}({h^{**}}_{(4,1)_{\kappa}})^{s_{2n+1}} \beta_{2n}\}$ is less than the image under the transformation $\{({h^{**}}_{(4,1)_{\kappa}})^{s_{2n+1}} \beta_{2n}\}$. This proves the subcase when $s_{2n+1} < 0$.  
          \\$(ii)$ Subcase II, for $s_{2n+1} >0$: To deduce this, we have studied two further subcases, one is when ${e} \le \lvert \beta_{2n} (x)\rvert < (\lambda + 1)$ and the other for $(\lambda -1) \le \lvert \beta_{2n}(x)\rvert < {e}$, where ${e} = \sqrt{\lambda^2- (1- {\kappa})^2}$ is the attractive fixed point of the transformation $({h^*}_{(3,2)_{\kappa}})$. 
         \\Subcase II(a): For the first subcase when ${e} \le \lvert \beta_{2n} (x)\rvert < (\lambda + 1)$, one can express the transformation $\beta'_{2t+2}(x)$ successively.
         \begin{eqnarray}
         \lvert \beta'_{2t+2}(x) \rvert &=& \Big \lvert \frac{d}{dx} \{ \beta_{2t+2}(x)\} \Big \rvert \notag \\
         &=& \Big \lvert \frac{d\{({h^*}_{(3,2)_{\kappa}})^{s_{2t+2}} ({h^{**}}_{(4,1)_{\kappa}})^{s_{2t+1}} \beta_{2t}(x)\}}{d\{({h^{**}}_{(4,1)_{\kappa}})^{s_{2t+1}} \beta_{2t}(x)\}} \Big \rvert \times \Big \lvert \frac{d\{({h^{**}}_{(4,1)_{\kappa}})^{s_{2t+1}} \beta_{2t}(x)\}}{d\{\beta_{2t}(x)\}} \Big \rvert \times \Big\lvert \frac{d\{\beta_{2t} (x)\}}{dx}  \Big\rvert \notag \\
         &=& \Big \lvert \frac{d\{({h^*}_{(3,2)_{\kappa}})^{s_{2t+2}} ({h^{**}}_{(4,1)_{\kappa}})^{s_{2t+1}} \beta_{2t}(x)\}}{d\{({h^{**}}_{(4,1)_{\kappa}})^{s_{2t+1}} \beta_{2t}(x)\}} \Big \rvert \times \Big \lvert \frac{d\{({h^{**}}_{(4,1)_{\kappa}})^{s_{2t+1}} \beta_{2t}(x)\}}{d\{\beta_{2t}(x)\}} \Big \rvert \times \lvert \beta'_{2t} (x)  \rvert \notag\\
         &<& \Big \lvert \frac{d\{({h^*}_{(3,2)_{\kappa}})^{s_{2t+2}} ({h^{**}}_{(4,1)_{\kappa}})^{s_{2t+1}} \beta_{2t}(x)\}}{d\{({h^{**}}_{(4,1)_{\kappa}})^{s_{2t+1}} \beta_{2t}(x)\}} \Big \rvert \times \Big \lvert \frac{d\{({h^{**}}_{(4,1)_{\kappa}})^{s_{2t+1}} \beta_{2t}(x)\}}{d\{\beta_{2t}(x)\}} \Big \rvert \notag \\
         &=& \Big \lvert \frac{d\{({h^*}_{(3,2)_{\kappa}})^{s_{2t+2}} ({h^{**}}_{(4,1)_{\kappa}})^{s_{2t+1}} \beta_{2t}(x)\}}{d\{({h^{**}}_{(4,1)_{\kappa}})^{s_{2t+1}} \beta_{2t}(x)\}} \Big \rvert \times \Big \lvert \frac{d\{({h^{**}}_{(4,1)_{\kappa}})^{s_{2t+1}} (x)\}}{dx} \Big \rvert, \text{ replacing } x \text{ by } \beta_{2t}(x) \notag\\
                  &=& \Big \lvert \frac{d\{({h^*}_{(3,2)_{\kappa}})^{s_{2t+2}} ({h^{**}}_{(4,1)_{\kappa}})^{s_{2t+1}} \beta_{2t}(x)\}}{d\{({h^{**}}_{(4,1)_{\kappa}})^{s_{2t+1}} \beta_{2t}(x)\}} \Big \rvert \times \Big \lvert \Big( \frac{(1-\kappa)^2}{\{x + (\lambda+2)\}} \Big)^{s_{2t+1}}\Big \rvert  \notag \\   
                 &<&  \frac{1}{\lvert x + (\lambda+2)\rvert^{s_{2t+1}}} \times\Big \lvert \frac{d\{({h^*}_{(3,2)_{\kappa}})^{s_{2t+2}} (y)\}}{dy} \Big \rvert, \text{ replacing } y \text{ by } \{({h^{**}}_{(4,1)_{\kappa}})^{s_{2t+1}} \beta_{2t}(x)\}  \notag\\ 
                 &<& \frac{1}{\lvert x + (\lambda+2)\rvert^{s_{2t+1}}} \times \Big \lvert \frac{(1-\kappa)^2  }{\Big\{\Big(\frac{(1-\kappa)^2}{\{x + (\lambda+2)\}}\Big)^{s_{2t+1}} + \lambda \Big \}^2} \Big \rvert \notag\\
                 &<& \frac{1}{\lvert x + (\lambda+2)\rvert^{s_{2t+1}}} \times  \frac{1}{\Big \lvert \Big(\frac{(1-\kappa)^2}{\{x + (\lambda+2)\}}\Big)^{s_{2t+1}} + \lambda \Big \rvert^2}  \notag\\
                 &<& 1, \text{ since outside the semi-circle } SC_{4,\kappa}, \text{ we have } \lvert \text{x} + (\lambda+2) \rvert > 1 - \kappa. \notag
         \end{eqnarray}
         
         So, we have reached the proof of the subcase II(a).
         \\Subcase  II(b): In this case, the given condition is $(\lambda -1) \le \lvert \beta_{2n}(x)\rvert < {e}$. So, this condition yields that $\lvert ({h^{**}}_{(4,1)_{\kappa}})^{s_{2t-1}} \beta_{2t-2} (x) \rvert$  $< (\lambda - 1)$. Therefore, $s_{2t-1} < 0$. Now, one can easily compute the value of  $\lvert \beta'_{2t-2}(x) \rvert$ by using the technique operated to evaluate the value of $\lvert \beta'_{2t+2}(x) \rvert$ in the subcase II(a). Then, the proof of the subcase II(b) follows. So, we have established that  $\lvert \beta'_{2t+2}(x) \rvert <1$. Consequently, we have finished the induction. 
          
            Hence, the proof of this lemma is done.
      \end{proof} 
      After establishing the preceding lemma, we are now prepared to derive Lemma \ref{l 2.3}. The subsequent lemma is also related to the lengths of components of the domain of discontinuity for the group $\Gamma^{N}_{S_{\kappa}}$ that intersects the upper vertical axis within the hyperbolic plane.
      
      \begin{lem} \label{l 2.3}
            The length of each component of $\{\Omega\textbf{(}{h^{**}}_{(4,1)_{\kappa}} \Gamma_{S_{\kappa}}^{N} ({h^{**}}_{(4,1)_{\kappa}})^{-1}\textbf{)} \cap i\mathbb{R}^+\}$ which meets the vertical fragment $i[0,$ $ (\lambda +2)(\lambda + 3)]$ is less than the following quantity  \begin{eqnarray}
             \Big\{\frac{2}{(4 \lambda^2 + 9\lambda + 3)} \times (\lambda +2) \times \sqrt{2.01} \times \sqrt{\kappa} \times \sqrt{6 \lambda^4 + 16 \lambda^3 + 3 \lambda^2 - 16 \lambda -8}\Big\}, \notag
            \end{eqnarray} where $\lambda = \lvert x \rvert > 1$ and $\lvert x \rvert$ denotes the modulus value of the number $x$.
            \end{lem}
            
        \begin{proof}
        It is obvious that, $\Big({h^{**}}_{(4,1)_{\kappa}}\Big)^{-1}(z) = \frac{( \lambda + 2)  z - \{( \lambda + 2)^2 - (1-\kappa)^{2}\}}{-z + (\lambda + 2)}$.\\ So, when $z$ $<$ $\frac{\lambda^2 + 5 \lambda + 5}{\lambda +3}$ (for $\kappa \rightarrow 0^+$), the value of the transformation $\Big({h^{**}}_{(4,1)_{\kappa}}\Big)^{-1}$ is less than $1$.\\
        Hence, the function $\Big({h^{**}}_{(4,1)_{\kappa}}\Big)^{-1}$ is contracting.\\ Now, for $\kappa \rightarrow 0^+$, we have the following. \begin{eqnarray}
        \Big({h^{**}}_{(4,1)_{\kappa}}\Big)^2 : z & \longrightarrow & \frac{\Big[(\lambda +2) - \frac{1}{2(\lambda+2)} \Big]z + (\lambda + 2)^2 -1 }{z + \Big[(\lambda +2) - \frac{1}{2(\lambda+2)} \Big]}.\notag
        \end{eqnarray}
        Note that, $\Big({h^{**}}_{(4,1)_{\kappa}}\Big
        )^2(z)$ $=$ $h(z)$, \begin{eqnarray} \text{ for }
        z= \frac{-(2\lambda^3 + 10 \lambda^2 -15\lambda +7) \pm \sqrt{4\lambda^6 + 56 \lambda^5 + 168 \lambda^4 +104 \lambda^3 + 885\lambda^2 + 1042\lambda + 433}}{4(\lambda +2)^2}, \notag
        \end{eqnarray} where, $h: z \rightarrow (\lambda +2)z$, where $z$ $\in$ $\mathbb{H}^2$. \\In fact, for simplicity, we want to go forward to the function $h$ instead of the transformation  $\Big({h^{**}}_{(4,1)_{\kappa}}\Big
                )^2$, without loss of generality.  Now, if we utilize the transformation ${h^{**}}_{(4,1)_{\kappa}}$ twice to the domain mentioned in Lemma \ref{l 2.2}, i.e., the fragment  $[-(\lambda + 3), -(\lambda + 1)] \cup [(\lambda + 1), (\lambda + 3)]$, the component $\{\Omega\textbf{(}{h^{**}}_{(4,1)_{\kappa}} \Gamma_{S_{\kappa}}^{N} ({h^{**}}_{(4,1)_{\kappa}})^{-1}\textbf{)} \cap i\mathbb{R}^+\}$ will touch the vertical segment  $i[1,$ $ (\lambda +2)(\lambda + 3)]$. Since the transformation $\Big({h^{**}}_{(4,1)_{\kappa}}\Big)^{-1}$ is contracting, we can use the function $\Big({h^{**}}_{(4,1)_{\kappa}}\Big)^{-1}$ several times to extend the standing piece $i[1,$ $(\lambda +2)(\lambda + 3)]$ to the whole upright portion, i.e., $i[0,$ $(\lambda +2)(\lambda + 3)]$. Observe that, this act doesn't affect the size of the interval. So the lemma follows.
        \end{proof}

    Next, we performed the consecutive lemma related to the distance  between the semi-circles in the set ${SC}^*_\kappa$ which we have acquired from the images of the semi-circles $SC_{1,\kappa}, SC_{2,\kappa}, SC_{3,\kappa}$, and $SC_{4,\kappa}$ applying by the group  $\Gamma^{N}_{S_{\kappa}}$.
   
   \begin{lem} \label{l 2.4}
      Let, $\epsilon > 0$. If $\lvert Y_{jk\kappa} - Y_{(j+1)k\kappa} \rvert < \epsilon$ for at least two values of $k \in \{1, 2,3\}$ then $Z_{j\kappa} < \Big [\Big( \frac{2\lambda + 5}{2\sqrt{2} \sqrt{\lambda +2}} + 1 \Big) \Big(16[(2\lambda +5) \{1 + \textbf{(} \lambda + 2 + \tau \textbf{)} +  \textbf{(} \lambda + 2 + \tau \textbf{)}^2\}] + 1\Big) +1 \Big]\epsilon$.
      \end{lem}
      
      \begin{proof}
      Let, ${C}_j = {x}_j + i.0.{y}_j$ and ${r}_j$ be the center and radius of the semi-circle ${SC}_{\kappa}^j$ respectively. Assume that, ${d}_j$ is the distance between the centers of the semi-circles ${SC}_{\kappa}^j$ and ${SC}_{\kappa}^{j+1}$ where ${C}_{j+1} = {x}_{j+1} + i.0.{y}_{j+1}$ and ${r}_{j+1}$ are the center and radius of the semi-circle ${SC}_{\kappa}^{j+1}$ respectively. Now we define a semi-circle $SC_\kappa$ which is concentric to the semi-circle ${SC}_{\kappa}^{j+1}$ and tangent to ${SC}_{\kappa}^j$ at a point $P$. Clearly, $P$ lies on the real axis. So, ${SC}_\kappa$ can be expressed as $\sqrt{(x-{x}_{j+1})^2 + y^2} = ({r}_j - {d}_j)$. We set $Y'_k = {SC}_\kappa \cap i^{k-1} \mathbb{R^+}$. Note that, $Y'_3 P$ is the diameter of the semi-circle ${SC}_\kappa$ (see, Figure $4$). Let, $k_1$ and $k_2$ denote the two values of $k$ for which $\lvert Y_{jk_i \kappa} - Y_{(j+1)k_i \kappa} \rvert < \epsilon$. Without loss of generality, we suppose that \begin{eqnarray}
       \angle Y'_2 {C}_{j+1} {C}_j \le \angle Y'_2{C}_{j+1} Y'_1 & \text{or} & \angle {SC}_{j+1} Y'_2 \le \angle P{C}_{j+1}Y'_2. \notag
      \end{eqnarray} Let, $\theta$ denote $\angle S{C}_{j+1}Y'_2$. First, we consider the case when $\frac{\pi}{2} \le \theta \le \pi$, then we comment on that when $0\le \theta \le \frac{\pi}{2}$. Now, considering the triangle $\triangle {S}{P}Y'_2$ with $\angle Y'_2 {P}{S} = 90^\circ$, we get \begin{eqnarray}
      \sin \theta > \sin \angle Y'_2 {S}{P} & \text{and} & \sin \angle Y'_2 {S}{P} = \frac{\lvert Y'_2 - {P} \rvert}{2({r}_j - {d}_j)}. \notag
      \end{eqnarray} 
      So \begin{equation}
          \sin \theta > \frac{\lvert Y'_2 - P \rvert}{2(r_j - d_j)}.
          \end{equation}
          Again, $\lvert Y'_1\rvert + \lvert Y'_2 \rvert + \lvert Y'_3 \rvert \ge $ Diameter of $SC_\kappa$. \begin{equation}
          i.e., \frac{1}{2(r_j - d_j)} \ge \frac{1}{\sum_{k=1}^{3} \lvert Y'_k \rvert}.
          \end{equation}
         Also, by triangle inequality we obtain \begin{eqnarray}
          \lvert Y'_{k_i} - P \rvert + \lvert  Y'_{k_j} - P \rvert & \ge & \lvert Y'_{k_i} - Y'_{k_j} \rvert \notag \\
          &>& \lvert Y'_{k_i}\rvert,
          \end{eqnarray}
         for $i\ne j$, where $i,j=1,2$ or $3$.\\ Combining $(3)$, $(4)$, and $(5)$ we have \begin{equation}
          \sin \theta > \frac{\lvert Y'_3\rvert}{2(\sum_{k=1}^{3} \lvert Y'_k \rvert)}.
          \end{equation}
            \begin{center}
                   \hspace*{-0.5cm}
                   \includegraphics[width=17.6cm, height=8.5cm]{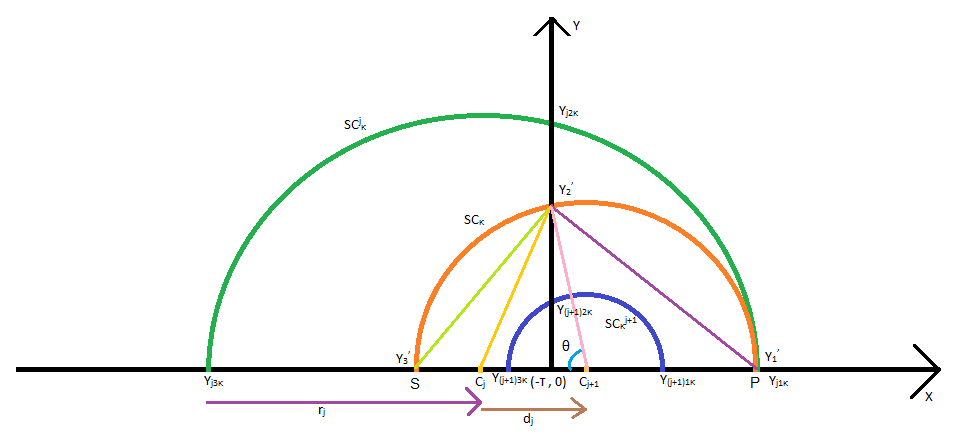} 
                   $Figure: 4$
                   \end{center}
            
            It is known that the fixed points of the transformation  ${h}_{(4,1)_{\kappa}}^{**}$ are $\tau$ and $- \tau$. Now, since the semi-circle ${SC}_\kappa$ separates the fixed points of the transformation ${h}_{(4,1)_{\kappa}}^{**}$, any image of the semi-circular curve ${SC}_\kappa$ under the map  ${h}_{(4,1)_{\kappa}}^{**}$ will also separate these fixed points. These fixed points are situated in either one of the curves in the set ${SC}^*_\kappa = \{{SC}^1_{\kappa}, {SC}^2_{\kappa}, {SC}^3_{\kappa}, ..., {SC}^{\textbf{I}^+}_{\kappa}\}$ or outside the semi-circle $SC_{1, \kappa}$. The image can not be inside ${SC}_{\kappa}^j$ (where $j = 1, 2,3, ..., \textbf{I}^+$). Also, the image can not be the semi-circle ${SC}_{\kappa}^j$ either, since ${SC}_{\kappa}^j$ is tangential to ${SC}_\kappa$ at the point $P$. Therefore it must be outside of the semi-circle ${SC}_{\kappa}^j$. So, the images of the $Y_k '$ for the map ${h}_{(4,1)_{\kappa}}^{**}$ must be outside the $Y_k '$ themselves. 
            
            Hence we get \begin{eqnarray}
            \lvert Y_{k+1}' \rvert &<& \Big(\frac{(\lambda + 2) z + \{(\lambda + 2)^2 - (1 - \kappa)^2\}}{z + (\lambda + 2)}\Big) \lvert Y_k ' \rvert \notag \\
            &<& \{ (\lambda + 2) - (-\tau )\} \lvert Y_k ' \rvert   \text{  (see, Section 2 for details)  } \notag \\
           \text{ i.e., } \lvert Y_{k+1}' \rvert &<& (\lambda + 2 + \tau) \lvert Y_k ' \rvert. \notag
            \end{eqnarray}  
            
            So \begin{eqnarray}
            \sum_{k=1}^{3} \lvert Y_k ' \rvert < \lvert Y_{k_1}' \rvert + \textbf{(} \lambda + 2 + \tau  \textbf{)}\lvert Y_{k_1}' \rvert + \textbf{(} \lambda + 2 + \tau  \textbf{)}^2  \lvert Y_{k_1}' \rvert .\notag
            \end{eqnarray}
            Now, equation $(6)$ yields
            \begin{eqnarray}
            \sin \theta > \frac{1}{\{1 + \textbf{(} \lambda + 2 + \tau \textbf{)} + \textbf{(} \lambda + 2 + \tau \textbf{)}^2 \}}. \notag
            \end{eqnarray}
             
             Therefore \begin{equation}
            (1+\cos \theta)^{-1} < 8 \{1 + \textbf{(} \lambda + 2 + \tau \textbf{)} + \textbf{(} \lambda + 2 + \tau \textbf{)}^2 \}^2.
            \end{equation} 
            
            Since the semi-circular curves ${SC}_{\kappa} ^j$ and ${SC}_{\kappa} ^{j+1}$, along with two other curves in the set $\{q(SC_{i,{\kappa}}): q\in \Gamma^{{N}}_{S_{\kappa}},$ $i=1,2,3, \text{ and }4\}$ bound a fundamental domain for $\Gamma_{S_{\kappa}}^{{N}}$, where $\Gamma_{S_{\kappa}}^{{N}}$ $=$  $< {h^*}_{(3,2)_{\kappa}} , {h^{**}}_{(4,1)_{\kappa}} >$, we assert that ${h^{**}}_{(4,1)_{\kappa}}(SC_{\kappa}^j) \cap SC_{\kappa}^{j+1} = \phi$. Let, $M$ and $N$ be two points lying on the semi-circle ${SC}_{\kappa} ^j$ (where $j = 1, 2,3, ..., \textbf{I}^+$). Note that the transformation ${h^{**}}_{(4,1)_{\kappa}}$  preserves any line through the points $\pm \tau$. Now, if we take a line ${H}_{(4,1)_{\kappa}}$ through the point $\{-\tau\}$ and the point $({x}_j, 0)$ we have the distance between the point $M$ and ${C}_j$ is $\{{r}_j - \textbf{(} \lvert {C}_j \rvert + \tau \textbf{)}\}$, which is away from the fixed point $\{-\tau \}$. Further, the distance between the point $N$ and ${C}_j$ is $\{{r}_j + \textbf{(} \lvert {C}_j \rvert + \tau \textbf{)}\}$. Also, the image of point $M$ under the transformation ${h^{**}}_{(4,1)_{\kappa}}$ situates on the line ${H}_{(4,1)_{\kappa}}$ and further from the point $\{-\tau \}$ than the point $N$. 
            
            This gives the following.
             \begin{eqnarray}
             \Big(\frac{( \lambda  + 2)z + \{ ( \lambda  + 2)^2 - (1 - \kappa)^2\}}{z + ( \lambda  + 2)} \Big) \textbf{(} {r}_j - (\lvert {C}_j \rvert + \tau ) \textbf{)} &>& \frac{1}{2} \textbf{(} {r}_j + ( \lvert {C}_j \rvert + \tau) \textbf{)}. \notag
             \end{eqnarray}
             Clearly for $\kappa \rightarrow 0^+,$ the subsequent result holds (see, section $1$). \begin{eqnarray}
             \Big(\frac{( \lambda  + 2)z + \{ ( \lambda  + 2)^2 - (1 - \kappa)^2\}}{z + ( \lambda  + 2)} \Big) &<& ( \lambda  + 2).\notag
             \end{eqnarray} So, by utilizing the above two consecutive results, we obtain \begin{equation}
            \lvert {C}_j \rvert = \lvert {x}_j \rvert < \Big\{ \frac{(2 \lambda  + 3)}{(2 \lambda + 5)} {r}_j - \tau \Big\}. 
            \end{equation} 
            Also \begin{eqnarray}
            {d}_j &\le& (\lvert {C}_j \rvert + \lvert {C}_{j+1} \rvert) \notag \\
           &\le& \{ \lvert {x}_j \rvert + \lvert {x}_{j+1} \rvert + 2 \tau \}.
            \end{eqnarray}
           Again, since the point $\{-\tau\}$ is located inside the semi-circle ${SC}_{\kappa} ^{j+1}$, it is also placed inside the semi-circle ${SC}'$. Hence we get \begin{eqnarray}
           \lvert {x}_{j+1} + \tau \rvert &\le& \{\lvert {x}_{j+1} \rvert + \lvert \tau \rvert \} \notag \\
           &<& \{ {r}_j - {d}_j\}.
           \end{eqnarray}
            Now, combining the equations $(8), (9),$ and $(10)$, we obtain the successive sequel. \begin{equation}
            {d}_j < 2 \Big(\frac{\lambda + 2}{2 \lambda + 5}\Big) {r}_j.
            \end{equation} 
            Let, $\epsilon' = \lvert Y_{jk_{i}{\kappa}} -Y'_{k_i} \rvert < \epsilon$, where $k_i = i$, for $i=1,2$, and $3$. Note that, for $k_i = 1, \epsilon' =0$ and $k_i = 2$ or $3$, $\epsilon' >0$. Clearly for $k_i =1, {r}_j = \lvert {C}_j - Y'_{k_i} \rvert + \epsilon'$ and when $k_i = 2$ or $3$, ${r}_j < \lvert {C}_j - Y'_{k_i} \rvert + \epsilon'$, where $\epsilon' >0$. Now, if we look at the circle at infinity, for $k_i =1$ and $k_i =3$, we get $({r}_j - \epsilon') = \lvert {C}_j - Y'_{k_1} \rvert$ and $({r}_j - \epsilon') < \lvert {C}_j - Y'_{k_3} \rvert$ respectively. Again, if we observe the triangle $\triangle Y'_{k_2} {C}_{j+1} {C}_j$, for $k_i = 2$, it gives $({r}_j - \epsilon') < \lvert {C}_j -Y'_{k_2} \rvert$. Now, applying the cosine formula on the triangle $\triangle Y'_{k_2} {C}_{j+1} {C}_j$, we have \begin{eqnarray}
            ({r}_j - \epsilon')^2 &<& \lvert {C}_j - Y'_{k_2} \rvert ^2 \notag \\
            &=& {d}_j ^2 + ({r}_j - {d}_j)^2 - 2{d}_j ({r}_j - {d}_j)\cos \theta. \notag 
            \end{eqnarray} So, for $k_i = i$, where $i=2$, we gain  
              \begin{eqnarray}
              (1 + \cos\theta)d_j &<& \frac{r_j\epsilon'}{r_j - d_j} \notag \\
              \text{ i.e., } {d}_j &<& {r}_j \epsilon ({r}_j -{d}_j)^{-1} (1+\cos\theta)^{-1}. 
              \end{eqnarray}
               Also, equation $(11)$ contributes \begin{eqnarray}
              ({r}_j - {d}_j)^{-1} < \frac{(2 \lambda  + 5)}{{r}_j}.\notag
              \end{eqnarray} 
               Putting the values of $(1+\cos \theta)^{-1}$ and $({r}_j - {d}_j)^{-1}$, equation $(12)$ reduces to the ensuing form. \begin{equation}
                 {d}_j < 8 (2 \lambda  + 5) \{1 + \textbf{(} \lambda  + 2 + \tau\textbf{)} + \textbf{(} \lambda + 2 + \tau\textbf{)}^2\} \epsilon.
                 \end{equation}
              Notice that, for the case $0\le \theta \le \frac{\pi}{2}$, the value of $(1+\cos \theta)^{-1}$ is trivially included in the equation $(7)$.
              \\ Now, in the following, we are aiming to evaluate the value of $Z_{j{\kappa}}$.
              \begin{eqnarray}
              Z_{j{\kappa}} &=& \max_k \{\lvert Y_{jk\kappa} - Y_{(j+1)k\kappa} \rvert\}\notag\\
              &<& \sum_{k=1}^{3} \lvert Y_{jk\kappa} - Y_{(j+1)k\kappa} \rvert\notag\\
              &=& 2 {r}_j - 2 {r}_{j+1} + \sqrt{{r}_j ^2 - \textbf{(}{x}_j + \tau \textbf{)}^2} - \sqrt{{r}_{j+1} ^2 - \textbf{(}{x}_{j+1} + \tau \textbf{)}^2}\notag\\
              &=& 2 {r}_j - 2 {r}_{j+1} + \frac{\{{r}_j^2 - ({x}_j + \tau )^2\} - \{{r}_{j+1}^2 - ({x}_{j+1} + \tau )^2\}}{\sqrt{{r}_j^2 - ({x}_j + \tau )^2} + \sqrt{{r}_{j+1}^2 - ({x}_{j+1} + \tau )^2}}.\notag
              \end{eqnarray}
           		From equation $(8)$ we earn
                  \begin{eqnarray}
                  \lvert {x}_j \rvert &<& \Big\{ \frac{(2 \lambda  + 3)}{(2 \lambda + 5)} {r}_j - \tau \Big\}. \notag
                  \end{eqnarray}
                     So 
                     \begin{eqnarray}
                     {r}_j^2 - ({x}_j + \tau)^2 &>& \frac{8( \lambda  + 2)}{(2  \lambda  + 5)^2} {r}_j^2. \notag
                     \end{eqnarray}
                     Hence 
                     \begin{eqnarray}
                     \frac{1}{\sqrt{{r}_j^2 - ({x}_j + \tau)^2 + {r}_j^2} - \sqrt{({x}_j + \tau)^2}} < \frac{(2 \lambda  + 5)}{2\sqrt{2}\sqrt{ \lambda  + 2} ({r}_j + {r}_{j+1})}. \notag
                     \end{eqnarray}
                      Therefore 
                  \begin{eqnarray}
                 Z_{j{\kappa}} &<& 2 {r}_j - 2 {r}_{j+1} + \frac{({r}_j^2 - {r}_{j+1}^2)+ \{\lvert {x}_j - {x}_{j+1} \rvert - 2 \tau\} ({r}_j + {r}_{j+1})}{\frac{2\sqrt{2} \sqrt{\lambda  + 2} ({r}_j + {r}_{j+1})}{(2 \lambda  + 5)}} \notag\\
                   &<& 2({r}_j - {r}_{j+1}) + \frac{(2 \lambda  + 5)}{2\sqrt{2}\sqrt{ \lambda  + 2}} \Big[ ({r}_j - {r}_{j+1}) + \lvert {x}_j - {x_{j+1}} \rvert - 2\tau \Big] \notag\\
                   &=& \Big[\frac{2  \lambda  + 5 + 4 \sqrt{2} \sqrt{ \lambda  + 2} }{2\sqrt{2}\sqrt{ \lambda  + 2}}  \Big] ({r}_j - {r}_{j+1}) + \frac{2 \lambda  + 5}{2\sqrt{2} \sqrt{ \lambda  + 2}} \Big [\lvert {x}_j - {x}_{j+1}\rvert - 2\tau \Big] \notag\\
                   &=& \frac{1}{2\sqrt{2}\sqrt{\lambda  + 2}}  \Big[({r}_j - {r}_{j+1}) \textbf{(} 4\sqrt{2}\sqrt{ \lambda  + 2} + 2 \lambda  + 5 \textbf{)} + (2 \lambda  + 5) \lvert {x}_j - {x}_{j+1}\rvert - 2(2 \lambda + 5) \tau \Big].  \notag
                 \end{eqnarray}
                 Now, if we take a line through the centers of the semi-circles ${SC}_{\kappa}^j$ and ${SC}_{\kappa}^{j+1}$ and utilize the equation $(13)$, we reach the subsequent outcome. \begin{eqnarray}
                 {r}_j - {r}_{j+1} < {d}_j + \epsilon.\notag
                 \end{eqnarray} Also, from our assumption we have \begin{eqnarray}
                 \{\lvert {x}_j - {x}_{j+1} \rvert - 2\tau \} < {d}_j.\notag
                 \end{eqnarray} Hence
                    \begin{eqnarray}
                    Z_{j{\kappa}} &<& \frac{1}{2\sqrt{2} \sqrt{ \lambda  + 2}} \Big[({d}_j + \epsilon)\textbf{(} 4\sqrt{2} \sqrt{\lambda  + 2} + 2 \lambda  + 5 \textbf{)} (2 \lambda  + 5){d}_j \Big] \notag \\
                    &=& \Big(\frac{2 \lambda  + 5}{2\sqrt{2} \sqrt{ \lambda  + 2}} + 1\Big) \Big(2{d}_j + \epsilon \Big) + \epsilon. \notag 
                    \end{eqnarray}
                    Finally, putting the value of ${d}_j$ $\textbf{(}$from equation (13)$\textbf{)}$ the above inequality converts to 
                   \begin{eqnarray}
                   Z_{j{\kappa}} < \Big[\Big(\frac{2 \lambda  + 5}{2\sqrt{2} \sqrt{\lambda  + 2}} + 1\Big) \Big(16[(2 \lambda  + 5) \{ 1 + \textbf{(}  \lambda  + 2 + \tau \textbf{)} + \textbf{(}  \lambda  + 2 + \tau \textbf{)}^2 \}] + 1 \Big) + 1 \Big] \epsilon, \notag
                   \end{eqnarray} 
                    which is our required value of $Z_{j{\kappa}}$.
                    
                    Therefore, the lemma is established.
      \end{proof} 
              
   For the value of sufficiently very small positive $\kappa$ and the upper bound of the gaps between the semi-circles in the set ${SC}^*_\kappa$, the following lemma will be useful in the sequel.
   
   \begin{lem} \label{l 2.5}
   For every sufficiently small positive number $\kappa < 10^{-11}$ and for every positive integer, where $1\le j < \textbf{I}^+,$ the value of $Z_{j\kappa}$ is less than $\frac{1}{5}$.
   \end{lem}
   
   \begin{proof} 
   Let, $F^j_{\Gamma^{N}_{S_{\kappa}}}$ (see, the light turquoise color region in Figure $5$) be a fundamental domain for the group $\Gamma_{S_{\kappa}}^{N}$. Suppose $V_{j{\kappa}}$ (see, light turquoise + purple color areas in the figure below) is the doubly-connected domain enclosed by the semi-circles $SC_{\kappa} ^j$ and $SC_{\kappa} ^{j+1}$ in the hyperbolic plane with the boundary $\mathbb{R} \cup \{ \infty \}$. Observe that, apart from the two semi-circles $SC_{\kappa} ^j$ and $SC_{\kappa} ^{j+1}$ there are also two semi-circles in the boundary of the fundamental domain $F^j_{\Gamma^{N}_{S_{\kappa}}}$. So, we are gaining exactly one segment viz., $G$ (see, the red vertical piece) of $\{V_{j\kappa} \cap i\mathbb{R}^+\}$ that is equal in the fragment $\{F^j_{\Gamma^{N}_{S_{\kappa}}} \cap i\mathbb{R}^+\}$, where $G$ is situated within the upper-half plane (between the two curves $SC_{\kappa} ^j$ and $SC_{\kappa} ^{j+1}$). Hence, $G$ is also lying on the component of $\{\Omega\textbf{(}{h^{**}}_{(4,1)_{\kappa}} \Gamma_{S_{\kappa}}^{N} ({h^{**}}_{(4,1)_{\kappa}})^{-1}\textbf{)} \cap i\mathbb{R}^+\}$ meeting the closed interval $i[0,$ $(\lambda +2)(\lambda + 3)]$. On the other hand, from Lemma \ref{l 2.3} we get that the length of the portion $G$ is less than the ensuing quantity \begin{eqnarray}
             \Big\{\frac{2}{(4 \lambda^2 + 9\lambda + 3)} \times (\lambda +2) \times \sqrt{2.01} \times \sqrt{\kappa} \times \sqrt{6 \lambda^4 + 16 \lambda^3 + 3 \lambda^2 - 16 \lambda -8}\Big\}. \notag
            \end{eqnarray}
    Now, if we set $\kappa < 10^{-11}$, then we obtain $Z_{j\kappa} < \frac{1}{5}$, which is the required value of $Z_{j\kappa}$. This proves the lemma.
   \end{proof}
     \begin{center}
       \hspace*{-0.5cm}
       \includegraphics[width=18cm, height=6.4cm]{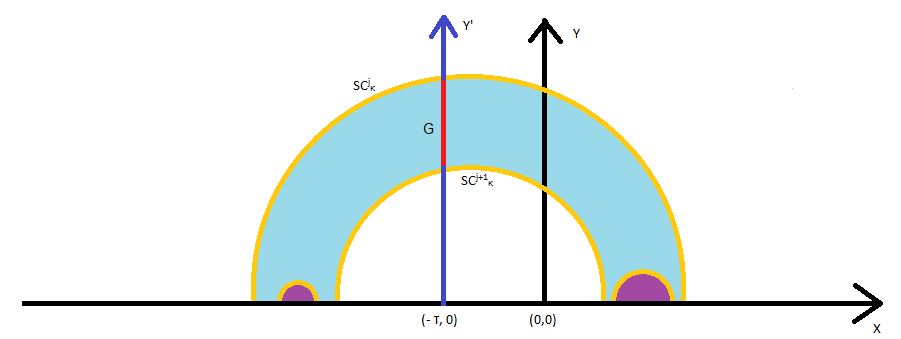} 
       $Figure: 5$
       \end{center}

      \section{\textbf{PROOF OF OUR MAIN THEOREM }}
     After establishing all the needful results in Section $2$, we are now ready to derive the main theorem of this paper.

       \begin{proof}[\textbf{Proof of Theorem \ref{t 1}.}]
        Recall that, we have intended to prove this theorem by contradiction.
                 Let, $SC_{\kappa^*}$ be a semi-circle meeting the closed interval $[-( \lambda + 2)(\lambda +3),$ $- ( \lambda  + 3)]$ which is analogous to the semi-circle $SC_{4,\kappa}$ for the group generated by $\Big(h_{(4,1)_{\kappa}} ^{**}\Big)^2$. The semi-circle $SC_{4,\kappa}$ separates the fixed points $\tau$ and $\{-\tau\}$ of the transformation ${h^{**}}_{(4,1)_{\kappa}}$. So, the semi-circle $SC_{\kappa^*}$ also separates these fixed points. Since the semi-circle $SC_{\kappa^*}$ is outside the semi-circle $SC_{3,\kappa}$, its image ${h^*}_{(3,2)_{\kappa}}(SC_{\kappa^*})$ meets the closed interval $[ (\lambda - 1),$ $(\lambda  + 1)]$ twice. Again, the fixed point $\tau$ and the point
                   $\Big\{-( \lambda + 2)(\lambda +3)\Big\}$ are outside of the semi-circle $SC_{\kappa^*}$. Hence, the points ${h^*}_{(3,2)_{\kappa}}(\tau)$ and  ${h^*}_{(3,2)_{\kappa}}\Big\{-( \lambda + 2)(\lambda +3) \Big\}$ will be separated by the semi-circular curve ${h^*}_{(3,2)_{\kappa}}(SC_{\kappa^*})$ and placed within this curve by situating on the boundary of the hyperbolic plane. Therefore, the diameter of ${h^*}_{(3,2)_{\kappa}}(SC_{\kappa^*})$ will be greater than $\frac{1}{5}$, since \begin{eqnarray}
                   \Big \lvert {h^*}_{(3,2)_{\kappa}} \Big\{-( \lambda + 2)(\lambda +3) \Big\} -
                                {h^*}_{(3,2)_{\kappa}}\Big(\tau \Big) \Big \rvert > \frac{1}{5}. \notag
                   \end{eqnarray} But Lemma \ref{l 2.5} shows that the gaps between the curves in the set ${SC}^*_\kappa$ are less than $\frac{1}{5}$. So, ${h^*}_{(3,2)_{\kappa}}(SC_{\kappa^*})$ intersects some of the semi-circular curves $SC_{\kappa} ^j$ (where $j = 1, 2,3, ..., \textbf{I}^+$) in the family ${SC}^*_\kappa$. Hence, all images of the semi-circles $SC_{1,\kappa}, SC_{2,\kappa},$
                   $SC_{3,\kappa}$, and $SC_{4,\kappa}$ are not disjoint, which is impossible. So we reach a contradiction. Therefore, the Schottky group $\Gamma^{N}_{S_{\kappa}}$  $=$  $<{h^*}_{(3,2)_{\kappa}}, {h^{**}}_{(4,1)_{\kappa}}>$ is a non-classical Fuchsian Schottky group. This completes the proof of Theorem \ref{t 1}.
       \end{proof}

       \section{\textbf{EXAMPLES OF FUCHSIAN SCHOTTKY GROUPS WITH NON-CLASSICAL GENERATING SETS}}
       In the previous sections, we have provided the structure of the rank $2$ Fuchsian Schottky groups with non-classical generating sets by establishing Theorem \ref{t 1}. Now, in Theorem \ref{t 1} if we put the value of $\lambda$ is equal to $2$ and $1.6666666667$ in the hyperbolic elements $(1)$ and $(2)$ in Section $1$, we get the succeeding four M\"obius transformations. 
       
       Let, $\mu$ $=$ $1.6666666667$.
        \begin{eqnarray}
        {h^1}_{(3,2)_{\kappa_1}} : z & \longrightarrow & \frac{2 (1-{\kappa_1})^{-1} z + (1-{\kappa_1})\{4 (1-{\kappa_1})^{-2} -1\}}{(1-{\kappa_1})^{-1} + 2 (1-{\kappa_1})^{-1}} \\{h^2}_{(4,1)_{\kappa_1}} : z & \longrightarrow & \frac{4 (1-{\kappa_1})^{-1} z + (1-{\kappa_1})\{16 (1-\kappa_1)^{-2} -1\}}{(1-\kappa_1)^{-1} + 4 (1-\kappa_1)^{-1}} \\
        {h^3}_{(3,2)_{\kappa_2}} : z & \longrightarrow & \frac{\mu (1-\kappa_2)^{-1} z + (1-\kappa_2)\{\mu^2 (1-\kappa_2)^{-2} -1\}}{(1-\kappa_2)^{-1} + \mu (1-\kappa_2)^{-1}} \\ \text{  and  }
        {h^4}_{(4,1)_{\kappa_2}} : z & \longrightarrow & \frac{(2+\mu) (1-\kappa_2)^{-1} z + (1-\kappa_2)\{(2+\mu)^2 (1-\kappa_2)^{-2} -1\}}{(1-\kappa_2)^{-1} + (2+\mu) (1-\kappa_2)^{-1}}
        \end{eqnarray}
       
      For $\lambda =2$, we examine in Theorem \ref{t 1} that when we take the value of $\kappa_1$ is less than $4 \times 10^{-12}$ then the Fuchsian Schottky group $\Gamma^{N}_{S_{\kappa_1}}$ generated by the two hyperbolic M\"obius transformations ${h^1}_{(3,2)_{\kappa_1}}$ and  ${h^2}_{(4,1)_{\kappa_1}}$ is non-classical in the upper-half plane. So Corollary \ref{c 1} is done. On the other hand, for $\lambda = 1.6666666667$, the group $\Gamma^{N}_{S_{\kappa_2}}$ generated by ${h^3}_{(3,2)_{\kappa_2}}$ and ${h^4}_{(4,1)_{\kappa_2}}$ is a Fuchsian Schottky group with non-classical generating sets when $\kappa_2 < 9 \times 10^{-12}$. Hence Corollary \ref{c 2} also follows.
      
      Therefore, these two corollaries (Corollary \ref{c 1} and \ref{c 2}) give rise to two non-trivial examples of the Fuchsian Schottky groups with non-classical generating sets in the hyperbolic plane. \\
      

\textbf{Acknowledgment:} The second author greatly acknowledges The Council of Scientific and Industrial Research (File No.: 09/025(0284)/2019-EMR-I), Government of India, for the award  of SRF.\\

\section{\textbf{DATA AVAILABILITY STATEMENT}}
Our manuscript has no associated data.

\section{\textbf{DECLARATIONS: CONFLICT OF INTEREST}}
The authors declare that they have no conflict of interest.

                         \end{document}